\documentclass{amsart}
\usepackage{amssymb,amsmath,latexsym}
\usepackage{amsthm}
\usepackage{fontenc}
\usepackage{amssymb}
\numberwithin{equation}{section}
\newtheorem{theorem}{Theorem}[section]
\newtheorem{corollary}{Corollary}[theorem]
\newtheorem{lemma}[theorem]{Lemma}

\DeclareMathOperator*{\Ei}{Ei}

\begin{document}
\author{Alexander E Patkowski}
\title{On Salem's Integral Equation and related criteria}

\maketitle
\begin{abstract}We extend Salem's Integral equation to the non-homogenous form, and offer the associated criteria for the Riemann Hypothesis. Explicit solutions for the non-homogenous case are given, which in turn give further insight into Salem's criteria for the RH. As a conclusion, we show these results follow from a corollary relating the uniqueness of solutions of the non-homogenous form with Wiener's theorem.\end{abstract}



\section{Introduction}
In 1953, Salem [7] presented a new criteria for the Riemann Hypothesis in the form of an integral equation. Namely, the Riemann hypothesis is true if and only if 
\begin{equation} \int_{\mathbb{R}}\frac{e^{-\sigma y}f(y)}{e^{e^{x-y}}+1}dy=0,\end{equation} has no bounded solution other than the trivial case $f(y)=0,$ when $\frac{1}{2}<\sigma<1,$ $\sigma\in\mathbb{R}.$ The proof used by Salem involves using a theorem of N. Wiener on the zeros of the Fourier transform [1, pg.139], and utilizing a Mellin transform for Riemann's zeta function $\zeta(s)$ [8] that is analytic in the critical region $0<\Re(s)=\sigma<1.$ A proof of (1.1) has been offered in [1, pg.141--142] and further criteria and related results have recently been established in [10]. The more general homogenous form of (1.1) is also detailed in [4, eq.(1)--(2)].
\par The related non-homogenous form of the integral equation (1.1) may be written in the general form
\begin{equation}\lambda_1f(x)=\lambda_2g(x)+\int_{\mathbb{R}}k(x-y)f(y)dy,\end{equation}
with constants $\lambda_i\ge0,$ $i=1,2.$ Appealing to Fourier transforms, (1.2) has the known solution [9, pg.303--304]:
\begin{equation} f(x)=\frac{1}{\sqrt{2\pi}}\int_{\mathbb{R}}\frac{\lambda_2G(y)e^{-ixy}dy}{\lambda_1-\sqrt{2\pi}K(y)},\end{equation}
where $$G(y):=\frac{1}{\sqrt{2\pi}}\int_{\mathbb{R}}g(y)e^{ivy}dy.$$ Here $g(y)\in L_2(\mathbb{R})$ and $k(y)\in L_1(\mathbb{R}),$ so that $f(y)\in L_2(\mathbb{R})$ (see [9, pg. 315, Theorem 148]).
\par Now we observe that (1.1) may be seen as the homogenous case $\lambda_1=0,$ $\lambda_2=0.$ As it turns out, the non-homogenous case $\lambda_1=0,$ $\lambda_2\neq0$ offers nice criteria for the Riemann Hypothesis as well as providing a new angle on Salem's criteria. 

\begin{theorem}\label{thm:thm1} Suppose that $h(y)\in L_2(\mathbb{R}),$ and $H(x)$ is the Fourier transform of $h(y).$ Then the Riemann hypothesis is equivalent to the claim that 
$$h(x)=e^{\sigma x}\int_{\mathbb{R}}\frac{e^{-\sigma y}}{e^{e^{x-y}}+1}\phi(y)dy$$ has a unique solution for each $h(x)$ in the region $\frac{1}{2}<\sigma<1.$ In particular,
$$\phi(x)=\frac{1}{2\pi}\int_{\mathbb{R}}\frac{H(y)e^{-ixy}dy}{\Gamma(\sigma+iy)(1-2^{1-\sigma-iy})\zeta(\sigma+iy)},$$ provided that
$$\frac{H(y)}{\Gamma(\sigma+iy)(1-2^{1-\sigma-iy})\zeta(\sigma+iy)}\in L_2(\mathbb{R}).$$
\end{theorem}
Notice that if $h(x)=0,$ then $H(y)=0,$ since  by [5, Corollary 8.5] $H(y)$ is one-one on $L_1(\mathbb{R}).$ Hence the uniqueness property $\phi(x)=0$ in turn implies Salem's criteria. This follows directly from the Fredholm Alternative Theorem [6, pg.55, Corollary 4.18], which we discuss further in our proofs section. \par Using other Mellin transforms for $\zeta(s)$ that are valid in the region $0<\Re(s)=\sigma<1,$ we obtain two further examples. Let $\{x\}$ denote the fractional part of $x,$ and let $\psi(x)=\frac{d}{dx}\log(\Gamma(x))$ be the digamma function.

\begin{theorem}\label{thm:thm2} Suppose that $h(y)\in L_2(\mathbb{R}),$ and $H(x)$ is the Fourier transform of $h(y).$ Then the Riemann hypothesis is equivalent to the claim that 
$$h(x)=e^{\sigma x}\int_{\mathbb{R}}e^{-\sigma y}\{e^{y-x}\}\phi(y)dy$$ has a unique solution for each $h(x)$ in the region $\frac{1}{2}<\sigma<1.$ In particular,
$$\phi(x)=\frac{1}{2\pi}\int_{\mathbb{R}}\frac{(\sigma+iy)H(y)e^{-ixy}dy}{\zeta(\sigma+iy)},$$ provided that
$$\frac{(\sigma+iy)H(y)}{\zeta(\sigma+iy)}\in L_2(\mathbb{R}).$$
\end{theorem}

\begin{theorem}\label{thm:thm3} Suppose that $h(y)\in L_2(\mathbb{R}),$ and $H(x)$ is the Fourier transform of $h(y).$ Then the Riemann hypothesis is equivalent to the claim that 
$$h(x)=e^{\sigma x}\int_{\mathbb{R}}e^{-\sigma y}(\psi(e^{x-y}+1)+y-x)\phi(y)dy$$ has a unique solution for each $h(x)$ in the region $\frac{1}{2}<\sigma<1.$ In particular,
$$\phi(x)=-\frac{1}{2}\int_{\mathbb{R}}\frac{\sin(\sigma+iy)H(y)e^{-ixy}dy}{\zeta(1-\sigma-iy)},$$ provided that
$$\frac{\sin(\sigma+iy)H(y)}{\zeta(1-\sigma-iy)}\in L_2(\mathbb{R}).$$
\end{theorem}
Lastly, we offer a simple example of a unique solution for the case $h(x)=\Ei(-e^{x})e^{\sigma x}-2\Ei(-2e^{x})e^{\sigma x},$ in Theorem 1.1. Here $\Ei(x)$ is the exponential integral [3, pg.883]. Let $M(x)=\sum_{n\le x}\mu(n)$ denote Merten's function [8, pg.370].\newline
{\bf Example:}\label{ex4} \it Assuming the Riemann Hypothesis, we have that the solution to
\begin{equation}\Ei(-e^{x})e^{\sigma x}-2\Ei(-2e^{x})e^{\sigma x}=e^{\sigma x}\int_{\mathbb{R}}\frac{e^{-\sigma y}}{e^{e^{x-y}}+1}\phi(y)dy,\end{equation}
when $\frac{1}{2}<\sigma<1,$ is given by
$$\phi(x)=-\frac{1}{2\pi}\int_{\mathbb{R}}\frac{e^{-ixy}dy}{(\sigma+iy)\zeta(\sigma+iy)}=-e^{\sigma x}M(e^{-x}),$$ for $x<0,$ and $0$ for $x>0.$ \newline
\rm
This result follows from the Mellin--Perron formula for $M(x)$ [8, pg.370], and The Mellin transform [3, pg.638, eq.(6.223)]

$$\int_{\mathbb{R^{+}}}\Ei(-\beta y)y^{s-1}dy=-\frac{\Gamma(s)}{s\beta^{s}},$$ where $\mathbb{R^{+}}=(0,\infty)$ and $\Re(\beta)\ge0,$ $\Re(s)>0.$
\section{Proofs of Theorems}
To establish Salem's criteria follows from our main theorem we will need to utilize the Fredholm alternative theorem regarding uniqueness of solution for the non-homogenous integral equation [6, pg.55, Corollary 4.18]. We state this theorem in a more restrictive form to appeal to our Fourier transform solutions, and only in one direction. However, in its full generality, this theorem is an "if and only if" statement.

\begin{lemma} (Fredholm Alternative Theorem) Suppose $k(x,y)\in L_1(\mathbb{R}),$ $\phi(x)\in L_2(\mathbb{R}).$ If the homogeneous equation
\begin{equation}\phi(x)=\int_{\mathbb{R}}k(x,y)\phi(y)dy,\end{equation} only has trivial solution $\phi(y)=0,$ then the non-homogenous equation
\begin{equation}\phi(x)=h(x)+\int_{\mathbb{R}}k(x,y)\phi(y)dy,\end{equation} has a unique solution for each $h(x)\in L_2(\mathbb{R}).$ Otherwise, if the equation (2.1) has a non-trivial solution the non-homogenous equation (2.2) has no solution or infinitely many solutions.

\end{lemma}

\begin{proof}[Proof of Theorem~\ref{thm:thm1}] First, observe that Salem's kernel satisfies $k(y)\in L_1(\mathbb{R}).$ Since we assumed $h(y)\in L_2(\mathbb{R})$ in the hypothesis, we have that $\phi(y)\in L_2(\mathbb{R}),$ provided that $H(y)/K(y)\in L_2(\mathbb{R})$ [9, pg.315 ,Theorem 148]. From [7], we know that
$$\int_{\mathbb{R}}\frac{e^{(\sigma +ix)y}}{e^{e^{y}}+1}dy=\int_{\mathbb{R^{+}}}\frac{y^{s-1}}{e^{y}+1}dy=\Gamma(s)(1-2^{1-s})\zeta(s)=K(\sigma+ix),$$
where $\mathbb{R^{+}}=(0,\infty),$ and $\Re(s)=\sigma>0.$ Hence the case $\lambda_1=0,$ $\lambda_2=-1,$ of (1.3) with $K(\sigma+ix)=\Gamma(s)(1-2^{1-s})\zeta(s),$ gives our solution. Since the Riemann Hypothesis [8, pg. 256] says that there are no nontrivial zeros of $\zeta(s)$ in the region $1>\sigma>\frac{1}{2},$ and $\Gamma(s)\neq0$ for all $s\in\mathbb{C},$ the equivalence to the Riemann Hypothesis follows.
\end{proof}

\begin{proof}[Proof of Theorem~\ref{thm:thm2}] The proof is identical to that of Theorem 1.1, with the exception that we use [8, pg.14, eq.(2.1.5)]
$$\int_{\mathbb{R}}\{e^{-y}\}e^{(\sigma+ix)y}dy=\int_{\mathbb{R^{+}}}\{\frac{1}{y}\}y^{s-1}dy=-\pi\frac{\zeta(s)}{s},$$ valid for $0<\sigma<1.$ 

\end{proof}

\begin{proof}[Proof of Theorem~\ref{thm:thm3}] The proof is identical to that of Theorem 1.1, with the exception that we use [8, pg.34, eq.(2.15.7)]
$$\int_{\mathbb{R}}(\psi(e^{y}+1)-y)e^{(\sigma+ix)y}dy=\int_{\mathbb{R^{+}}}(\psi(y+1)-\log(y))y^{s-1}dy=-\pi\frac{\zeta(1-s)}{\sin(\pi s)},$$ valid for $0<\sigma<1.$ Additionally, we apply the fact that non-trivial zeros of the Riemann zeta function (i.e. $\rho$) appear in pairs (i.e. $\zeta(1-\rho)=0$) [8, pg.30].
\end{proof}

\begin{section}{Observations and Concluding remarks}
Here we have transcribed Salem's criteria in the context of Fredholm theory by establishing its relationship to the non-homogenous form. Indeed, we may relate Wiener's theorem (see [2, Chapter 6]) in the following way. Suppose 
that $K(y)$ is of the form $\zeta(\sigma+iy)w(\sigma+iy),$ with $w(\sigma+iy)\neq0$ for $\frac{1}{2}<\sigma<1,$ and assume the RH. Then Wiener's theorem, together with the Fredholm alternative theorem, says that the non-homogenous equation (1.2) with $\lambda_1=0,$ $\lambda_2\neq0,$ has a unique solution $f$ (for each $g$), if and only if the linear span of translates of $k(x)\in L_1(\mathbb{R})$ is dense in $L_1(\mathbb{R}).$ In other words, our theorems are really a special case of the following.

\begin{corollary} Suppose that $k(x)\in L_1(\mathbb{R}).$ The non-homogenous equation (1.2) $\lambda_1=0,$ $\lambda_2\neq0,$ has a unique solution solution $f$ (for each $g$) if and only if the linear span of translates of $k(x)\in L_1(\mathbb{R})$ is dense in $L_1(\mathbb{R}).$
\end{corollary}

\begin{proof} We show the reverse direction and leave the other direction to the reader. Suppose the linear span of translates of $k(x)\in L_1(\mathbb{R})$ is dense in $L_1(\mathbb{R}).$ Then Wiener's theorem says that $K(y)\neq0.$ Hence taking the Fourier transform of the convolution equation $k(x)\star f(x)=f(x)\star k(x)=0,$ tells us that we must have $f(x)=0.$ Hence, by [6, pg.55, Corollary 4.18] the result follows.
\end{proof}
Lastly, we touch on Salem's observation regarding the linear combination of translates of $k(x)$ as an approximation in $L_1(\mathbb{R})$ of $f(x)\star k(x).$ Recall [5, pg.346] that the translation operator is given by $T_{y}(k)(x)=k(x-y).$ Let $\epsilon>0$ be given and assume $f(x)\in L_1(\mathbb{R}).$ Then, according to [5, Corollary 6.17] together with our non-homogenous Fredholm equation (equation (1.2) with $\lambda_1=0,$ $\lambda_2=-1,$), there are numbers $z_i\in\mathbb{R},$ that satisfy
$$\left\lVert g(x)-\sum_{i}^{m}a_ik(x-z_i)\right\rVert_p<\epsilon,$$
where $a_i$ are constants and $p$ is either $1$ or $2.$
\end{section}

1390 Bumps River Rd. \\*
Centerville, MA
02632 \\*
USA \\*
E-mail: alexpatk@hotmail.com, alexepatkowski@gmail.com
\end{document}